\documentclass[12pt,oneside,reqno]{amsart}

\usepackage[a4paper,left=33mm,top=35mm,right=33mm,bottom=38mm]{geometry}

\usepackage{amsfonts}
\usepackage{amssymb}
\usepackage{graphicx}
\usepackage{ragged2e}
\usepackage{amsthm}
\usepackage{amsmath}
\usepackage[mathscr]{euscript}
 \let\mathscr\relax
\usepackage[scr]{rsfso}
\usepackage{siunitx}
\usepackage{dsfont}
\usepackage{mathtools}
\usepackage[pdftex,bookmarksnumbered,bookmarksopen]{hyperref}
\hypersetup{hidelinks = true}

\numberwithin{equation}{section}

\newtheorem{thm}{Theorem}[section]
\newtheorem{cor}[thm]{Corollary}
\newtheorem{prop}[thm]{Proposition}
\newtheorem{lem}[thm]{Lemma}
\newtheorem{conj}[thm]{Conjecture}

\theoremstyle{definition}
\newtheorem{defn}[thm]{Definition}

\theoremstyle{remark}
\newtheorem{rem}[thm]{Remark}

\begin{document}
	\title{Least primitive root and simultaneous power non-residues}
	\author{Andrea Sartori}
	\address{ London School of Geometry and Number Theory, King's College London, Strand, London WC2R 2LS, England, Uk }
	\email{andrea.sartori.16@ucl.ac.uk}
	\maketitle
		\begin{abstract}
		Let $p$ be a prime and let $g(p)$ be the least primitive root modulo $p$. We prove that for any $\epsilon>0$ and $p$ large enough the bound
	\begin{align}
	g(p)\ll p^{\frac{1}{4\sqrt{e}}+\epsilon} \nonumber
	\end{align}
	holds for most prime $p$ such that $p-1$ does not have small prime factors, but $2$. We also give an explicit description of the exceptional set.   
\end{abstract}
	\section{Introduction}
	\subsection{Vinogradov's conjecture}
Let $p$ be a prime, we denote by  $g(p)$ the least primitive root modulo $p$  or in other words the smallest generator of the multiplicative group of integer modulo $p$. Vinogradov  \cite{V1,V2}  conjectured the following: 
\begin{conj}
	Let  $\epsilon>0$ then for all large enough primes $p$ one  has $g(p)\ll p^{\epsilon}$.
\end{conj}
In \cite{V1,V2} he also gave the first non-trivial upper-bounds towards his conjecture, namely 
\begin{align}
 g(p)\ll 2^{\omega(p-1)+1}\frac{p-1}{\varphi(p-1)}p^{\frac{1}{2}} \nonumber
\end{align}
where $\omega(p-1)$ is the number of distinct prime factors of $p-1$ and $\varphi(p-1)$ is the Euler's totient function. In \cite{B1,B2} Burgess  made further progress proving that for any fixed $\epsilon>0$ and $p$ large enough we have
\begin{align}
 g(p)\ll p^{1/4+ \epsilon}. \nonumber
\end{align}
This bound is  the best unconditional result  known up-to-date. On the other hand under the Generalised Riemann Hypothesis, Ankeny \cite{A}, developing on the work of Linnik \cite{L1}, proved that 
\begin{align}
	g(p)\ll 2^{\omega(p-1)}\log^2p. \nonumber \nonumber
\end{align}
 Moreover, Vinogradov's conjecture is also known to be true for all but a small subset of the  primes:  Martin \cite{M} proved that for any $\epsilon>0$ there exist some constant $C>0$ such that bound $g(p)\ll \log^Cp$ holds for all primes up to $x$ with at most $O(x^{\epsilon})$ exceptions. Unfortunately, such result is purely \textquotedblleft existential" meaning  that, given a prime $p$, one cannot tell weather it belongs to the exceptional set or not. This motivated the search for improvements given some extra informations about $p$ as for example the prime factorization of $p-1$. In this spirit, Ford, Garaev and Konyagin \cite{FGK} proved that if $p-1$ has less then $\log\log p/2$ prime factors then we have the sharper bound $g(p)=p^{1/4-o(1)}$.  We follow such interest in proving our main result. 
	\subsection{Statement of the main result}
	To ease the exposition it will be useful to adopt the following piece of notation: we pick a very large parameter $x$ (which may be different at every occurrence) and denote by  $\mathbb{P}_x$ the discrete uniform measure on the primes up to $x$, and by $\mathbb{E}_x$ its expectation; moreover,  simply to avoid redundant sentences, we also make the following definition: 
	\begin{defn}
		We say that a prime $p$  is $y$-rough if  $p-1$ has no prime factors (but 2) smaller then $y$. 
	\end{defn}
We then prove the following: 
	\begin{thm}
		\label{Main theorem}
		For every $\epsilon>0$ and $\delta>0$  there exist some (large) $y=y(\epsilon,\delta)>0$ such that 
		\begin{align}
			\mathbb{P}_x\left(g(p)\ll p^{\frac{1}{4\sqrt{e}}+ \epsilon}\big\rvert\text{$p$ is $y$- rough}\right)>1-\delta .\nonumber
		\end{align}
			Moreover the exceptional set consists of those  primes which do not satisfies one of the following two conditions: 
		\begin{enumerate}
			\item uniformly on the divisors $d|p-1$ we have $j(d)\leq  10 \omega(d)$, 	where $j(d)$ is the Jacobsthal function \footnote{The constant 10 here has been chosen for simplicity of exposition and could be replaced with any constant bigger then 1 without affecting the proof.}.
			\item the following bound holds 
			
			\begin{align}
			\sum_{\substack{q|p-1 \\ q\neq 2 \text{prime}}}\frac{\log\log q}{\log q}\leq \epsilon/20 \nonumber.
			\end{align}
		\end{enumerate}
	\end{thm}
Conditions $(1)$ and $(2)$ come out naturally from our method and so we differ commenting on them to later. On the other hand,  we observe that, since $y$ only depends on $\epsilon$ and $\delta$, a simple inclusion-exclusion argument gives:
\begin{cor}
	\label{Main Corollary}
	For every $\epsilon>0$ there exists some $c=c(\epsilon)>0$ and $y=y(\epsilon)>0$ such that 
	\begin{align}
		\mathbb{P}_x\left( \text{$p$ satisfies conditions $(1)$, $(2)$ and is $y$-rough}\right)>c \nonumber
	\end{align}
	and so 
	\begin{align}
	\mathbb{P}_x\left( g(p)\ll p^{\frac{1}{4\sqrt{e}}+ \epsilon}\right)>c. \nonumber
	\end{align}
\end{cor}
\subsection{Outline of the proof}
Given a prime $p$ such that $p-1$ has $r$ prime factors $q_1,...,q_r$ say, we can think of the least primitive root modulo $p$ as the smallest simultaneous $q_1,...,q_r$-th power non-residue. Our strategy is then to first bound each $q_i$-th power non-residue individually and then use them to explicitly construct a $q_1,...,q_r$-th power non-residue of sufficiently small size. More precisely, in section 2 we  invoke an old argument of Vinogradov to bound the smallest $q_i$-th power non-residue (actually we prove a more general bound on the divisors $d|p-1$). In section \ref{Ssimultaneous} we then explicitly construct and element which is simultaneously $q_1,...,q_r$-th power non-residue. In  Section \ref{SJacobsthal} we deal with the technicalities to  show that the aforementioned construction gives an element which is \textquotedblleft generally \textquotedblright not worst then the product of the single $q_i$-th power non-residue provided we have some control on the prime factors of $p-1$, which manifests itself as condition $(1)$ in the Theorem. Finally in section \ref{SMaintheorem} we combine the above ingredients and the technical condition $(2)$ to prove our main Theorem. 
\subsection{Notation} 
The notation will be pretty much standard. We let $x$ be some asymptotic parameter going to infinity and we say that the quantity $X$ and $Y$ (depending on $x$)  satisfy $X\ll Y$ , $X\gg Y$ if there exist some constant $C$ independent of $x$ such that $X\leq C Y$ and $X\geq CY$ respectively, if both $X\ll Y$ and $X\gg Y$ (where the constant might be different ) we say $X \asymp Y$. We also write $O(X)$ for some quantity bounded in absolute value by a constant times $X$ and $X=o(Y)$ if $X/Y\rightarrow 0$ as $x$ goes to infinity, in particular we denote by $o(1)$ any function that tends to $0$ as $x\rightarrow \infty$. Moreover $p,q$ will always denote prime numbers (in general we have $q|p-1$ unless otherwise specified) and for a set of prime $A$ we denote by $P(A)=\prod_{p\in A}p$. Furthermore we adopt the convention $\log\log x=\max\{1,\log\log x\}$. Finally we let $\pi(x)$ be the number of prime numbers up to $x$ and for an integer $n$ we denote by $\omega(n)$, $\varphi(n)$, $j(n)$ be the number of distinct prime factors of $n$ and the Euler totient function, and the Jacobsthal function respectively. 

	\section{Smallest $d$-th power non-residue}
	\label{SVinogradov}
	Let $d$ be a divisor of $p-1$ where $p$ is a large prime,  we denote by  $\mathbf{g}(d)$ the smallest $d$-th power non residue modulo $p$, we follow an old argument of Vinogradov \cite{V3} to bound  $\mathbf{g}(d)$: 
	\begin{prop}[Vinogradov]
		\label{Vin}
		Let $p$ be a large  prime and let $d$ be a divisor of $p-1$, we then have that 
		\begin{align}
		\mathbf{g}(d)\ll p^{\frac{1}{u(d)}} \nonumber
		\end{align}
		where $u(d)$ is the unique solution to $\rho(u)=1/d$, see (\ref{defrho}). Moreover if $d$ is uniformly bounded by some constant we have for all $\epsilon>0$ the sharper bound 
		\begin{align}
		\mathbf{g}(d)\ll p^{\frac{1}{4u(d)}+ \epsilon}. \nonumber
		\end{align}
	\end{prop}
	 The original argument make an \textquotedblleft implicit\textquotedblright  use of smooth numbers, we will instead use the following inequality, see \cite[Page 274]{G}.   
	\begin{lem}
		Let $x,y$ be two positive real number and let $\Psi(x,y)$ the number of $y$-smooth numbers up to $x$, i.e. 
		\begin{align}
			\Psi(x,y)=|\{n\leq x: p|n\Rightarrow p<y\}| \nonumber
			\end{align}
			and let $u=\log x/\log y$. Then for $u$ or $y$ sufficiently large we have that 
			\begin{align}
		\Psi(x,y)\geq x\rho(u) \label{inePsi}
		\end{align}
			where $\rho(u)$ is the Dickman-De Bruijn function defined by 
			
			\begin{align}
			\begin{cases}\rho(u)=1 & 0\leq u \leq 1\\
			u\rho'(u)+\rho(u-1)=0 &u>1
			\end{cases} \label{defrho}
			\end{align}
	\end{lem}
We will need also the following property of $\rho(\cdot)$. 

	\begin{lem}
	\label{rhodecre}
	The function $\rho(u)$ is positive and decreasing for all $u\geq 0$. 
\end{lem}
\begin{proof}
	By definition is enough to show that $\rho(u)$ is positive. If $u\leq 1$ the lemma follows by (\ref{defrho}), if $u\geq 1$ then $u\in[n,n+1]$ for some positive integer $n$.  We prove the lemma inductively on $n$, so we can assume that $\rho(t)>0$ and $\rho'(t)<0$ for all $t\in[n-1,n]$. Integrating (\ref{defrho}) we obtain 
	\begin{align}
	\rho(u)-\rho(u-1)=-\int_{u-1}^{u}\frac{\rho(t-1)}{t}dt. \nonumber
	\end{align}
	By induction hypothesis we have $\rho(t-1)<\rho(u-1)$  and therefore 
	\begin{align}
	&\rho(u)-\rho(u-1)> -\rho(u-1)(\log(u)-\log(u-1)) \nonumber 
	\end{align}
	which, upon rearranging, implies the lemma. 
\end{proof}	
	 Finally we will make use of the following well-known inequality due to Burgess \cite{B2}. 
	\begin{lem}
		\label{Burgess}
		Let $\alpha>0$ be some small fixed quantity then there exists some $\delta=\delta(\alpha)>0$ such that for all large primes $p$ and non-principal characters $\chi$ mod $p$ we have 
		\begin{align}
			\left|\sum_{n\leq H}\chi(n)\right|\ll H p^{-\delta} \nonumber
		\end{align}
		for all $H>p^{1/4+\alpha}$. 
	\end{lem} 
We are now ready to prove the Proposition: 
	\begin{proof}[Proof Proposition \ref{Vin}]
		Let $d$ be given, on one hand the number of $d$-th power residues modulo $p$ is $(p-1)/d$ and on the other hand  any integer smaller then $\mathbf{g}(d)$ is a $d$-th power non-residue. So any  $\mathbf{g}(d)$-smooth number  is a $d$-th power residue but there can be at most $(p-1)/d$  of them, thus using (\ref{inePsi}) and taking $p$ large enough  we deduce that
		\begin{align}
			\frac{p-1}{d}\geq \Psi(p, \mathbf{g}(d))\geq \rho(u)p \label{Vin1}
		\end{align}
		where $u=\log p/ \log \mathbf{g}(d)$. Since $\rho$ is a strictly decreasing function of $u$ if $\mathbf{g}(d)>100 p^{1/u(d)}$ say then $\rho(u)$ would be strictly bigger then $1/d$ by a small but fixed amount and therefore this would contradict ({\ref{Vin1}}). This proves the first claim. 
		
		Let us now suppose that $d$ is absolutely bounded in term of $p$, we will deduce the second claim injecting Lemma \ref{Burgess} in the above argument. Firstly by orthogonality of characters modulo $p$ we can write the following indicator function on $d$-th power residues: 
		
		\begin{align}
	\frac{1}{d}\sum_{\chi^d=\chi_0}\chi(n)=\begin{cases}
	1 & n \hspace{1mm}\text{is a $d$-th power resdue} \\
	0 & n \hspace{1mm}\text{is a $d$-th power non-resdue}
	\end{cases} \nonumber
	\end{align}
	where $\chi_0$ is the principal character. Therefore for any $H>1$ we have that the number of $d$-th power residues in $[1,H]$ is given by 
	
	\begin{align}
		\sum_{n\leq H} \frac{1}{d}\sum_{\chi^d=\chi_0}\chi(n)= \frac{H}{d} + \frac{1}{d}\sum_{\substack{\chi^d=\chi_0\\ \chi\neq \chi_0}} \sum_{n\leq H} \chi(n) .\label{Vin2}
	\end{align} 
		Now take $H=p^{1/4+\alpha}$  for some $\alpha>0$ to be chosen later, then we can use Lemma \ref{Burgess} to bound the second term in (\ref{Vin2}) and find that
		
		\begin{align}
			|\{n\leq H:n \text{ is a $d$-th power residue }\}|= \frac{H}{d} + O(Hp^{-\delta}) \nonumber
		\end{align}
		where $\delta=\delta(\alpha)$ is given by Lemma \ref{Burgess}. Repeating the above smooth numbers argument lead us to the following inequality 
		\begin{align}
		\frac{H}{d} + O(Hp^{-\delta})\geq \Psi(H,\mathbf{g}(d))\geq \rho(u)H \label{Vin3}
		\end{align}
		where $u=\log H/\log \mathbf{g}(d)$. Now, let $\epsilon>0$ be given and suppose that $ \mathbf{g}(d) > 100 p^{1/4u(d)+ \epsilon}$ then $u>u(d)(1+ 4\alpha)(1+4\epsilon u(d) + O (\epsilon^2u^2(d)))$ . Since $d$ is absolutely bounded  and $\rho$ is strictly decreasing we can choose $\alpha=\epsilon$, so that	$	\rho(u)\geq \frac{1}{d} + O_{\epsilon}(1)$, inserting in (\ref{Vin3}) and dividing by $H$ we obtain
		
		\begin{align}
			\frac{1}{d} + O(p^{-\delta})\geq \frac{1}{d} + O_\epsilon(1) \nonumber
		\end{align}
		a contradiction for $p$ large enough in terms of $\epsilon$.  This proves the second claim and thus the Proposition.
	\end{proof}

\begin{cor}
	\label{corVin}
	Let $n$ be an integer then $u(n)$ is strictly decreasing and for $n$ large enough we also have the asymptotic expansion
		\begin{align}
	\frac{1}{u(x)}= \left[\log\log x - 1 + O\left(\frac{\log\log\log x}{\log\log x}\right)\right] \Bigg/ \log x \nonumber
	\end{align}
\end{cor}
\begin{proof}
	Since $\rho(\cdot )$ is a strictly decreasing function we immediately have that also $u(\cdot)$ must be strictly decreasing. For the second claim we have that $u(n)$ is defined by the equation $\rho(u(n))=1/n$, so if $n$ is large enough we can use the asymptotic expansion due to De Bruijn \cite{DB}  
		\begin{align}
	\rho(u)=\exp\left(-u\log u-u\log\log u + u +O\left(u\frac{\log\log(u)}{\log(u)}\right) \right). \nonumber
	\end{align}
	to see that 
	\begin{align}
		-u\log u-u\log\log u+u +O\left(u\frac{\log\log u}{\log u}\right)=-\log n \label{Cor1}
	\end{align} 
	where $u=u(n)$. Taking logarithm once more we have that
	\begin{align}
	\log\log x= \log u +\log\log u +O\left(\frac{\log\log u}{\log u}\right)  \label{19}
	\end{align}
	which implies that $\log u \geq \log\log n/2 $, say. Taking logarithm of (\ref{19}) we see that $\log \log u \leq 100 \log\log\log n$. Therefore the error term in (\ref{Cor1}) is $O(\log\log\log n/\log\log n)$ and the corollary follows inserting (\ref{19}) in (\ref{Cor1}) and dividing by $u$ both sides. 
\end{proof}
\begin{rem}
	\label{rem1}
	It will be useful to notice that $u(2)=\sqrt{e}$, see for example \cite{DE}. Form this we can deduce Burgess' bound on the smallest quadratic non-residue $\mathbf{g}(2)\ll p^{1/4\sqrt{e}+o(1)}$.  
\end{rem}

\section{Simultaneous power non-residue}
\label{Ssimultaneous}
In this section we explicitly construct an element mod $p$ which is not a $q$-th power residue for any $q|p-1$ as a product of $\mathbf{g}(q)$'s. Before embarking in this construction we need two preliminary ingredients. The first one is the following observation: let $q_1,...,q_r$ be the prime divisors of $p-1$ ordered such that
\begin{align}
	\mathbf{g}(q_1)\leq \mathbf{g}(q_2)\leq...\leq  \mathbf{g}(q_r) \nonumber
\end{align}
then, grouping together different elements, we can find disjoint sets $A_1,...,A_s$ such that 
\begin{align}
\mathbf{g}(A_1)< \mathbf{g}(A_2)<...< \mathbf{g}(A_r)	\label{A's}
\end{align}
where we denote by $\mathbf{g}(A_i)$ is the smallest element which is not a $q$-th power residue for all $q\in A_i$. Then the $\mathbf{g}(A_i)$ have the following property:  
\begin{rem}
	\label{rem2}
 for any $i>j$ we have that $\log (\mathbf{g}(A_j))\equiv 0 \pmod{q}$ for all $q\in A_i$, where $\log(\cdot)$ denotes the discrete logarithm mod $p-1$ with respect to some fixed primitive root. 
\end{rem}
The second one is a straightforward property of the  Jacobsthal function $j(\cdot )$ which, for a positive integer $n$,  is defined as the least integer $m$ such that any sequence of $m$ elements contains at least one co-prime to $n$. From the definition in fact one can directly deduce that:  
\begin{lem}
	\label{Jacobsthal}
	Let $A$ be a set of primes, $\{a(p)\}$ be a set of residue classes, one for each prime $p\in A$, and $m$ be the smallest integer such that $m\not \equiv a(p) \pmod{p}$ for all $p$. Then we have  $m\leq j(P(A))-1$, where $P(A)=\prod_{p\in A}p$.
\end{lem}
\begin{proof}
	By the Chinese Reminder Theorem there exists some $b \pmod{P(A)}$ such that $b\equiv -a(p) \pmod{p}$ for any $p\in A$ and $a(p)$ given above. Then for any $n\in[0,m)$ we have that 
	\begin{align}
	b+n\equiv -a_p+a_p\equiv 0 \pmod{p}. \nonumber
	\end{align}
	So the largest gap between integer co-prime to $P(A)$ is at least $m$, or in other words
	\begin{align}
	m\leq j(P(A))-1. \nonumber
	\end{align}
\end{proof}
We are now ready to start with the aforementioned construction. 
\begin{lem}
	\label{mainlemma}
	Let $p$ be a prime and let $A,B$ be two disjoint sets of prime divisors of $p-1$. Suppose that $y$ is a $q_A$-th power non-residue for all $q_A\in A$ and that $z$ is $q_B$-th power non-residue for all $q_B\in B$. If $y$ is a $q_B$ power residue for all $q_B\in B$  then there exist some  $m\in [1,j(P(A))-1]$ such that 
	
	\begin{align}
	y^m \cdot z \nonumber
	\end{align}
	is a $q$-th power non-residue for all $q\in A \cup B$.
\end{lem}
\begin{proof}
	We observe that for each $q_A\in A$ there exist an unique residue class $a(q_A)$ such that
	\begin{align}
		y^{a(q_A)}\cdot z\equiv 0 \pmod{q_A}. \label{mlem1}
	\end{align}
	Indeed, taking logarithm of both sides of (\ref{mlem1}) we see that the only solution is given by ${a(q_A)}\equiv \log z/\log y \pmod{q_A}$ as $\log y\not \equiv 0 \pmod{q_A}$ by assumption. It follows by Lemma \ref{Jacobsthal} that there exists some integer $m\in [0,j(P(A))-1]$ such that 
	\begin{align}
		m \not \equiv a(q_A) \pmod{q_A} \nonumber
	\end{align}
	for all $q_A\in A$. And so $y^m\cdot z$ is not a $q_A$-th power residue. Finally for any $q_B\in B$, in light of the fact that $y$ is a $q_B$-th power residue, we have that 
	\begin{align}
		m\log y + \log z\equiv \log z \not \equiv 0 \pmod{q_B} \nonumber
	\end{align}
 as required.  
\end{proof}
Applying the above lemma recursively we obtain: 
\begin{cor}
	\label{Cor 2.1}
	Let $p$ be a  prime and let $A_1,A_2,...A_s$ be as in (\ref{A's}) then there exists integers $m_i\in [1,j(P(A_i))-1]$ such that 
	\begin{align}
		y= \prod_{i=1}^{s}\mathbf{g}(A_i)^{m_i}. \nonumber
	\end{align}
	is not a $q$-th power for any $q|p-1$. Moreover we can take $m_s=1$.
\end{cor}
\begin{proof}
We proceed by constructing the $m_i$ inductively: by Remark \ref{rem2} $\mathbf{g}(A_s)$ is the only element among the $\mathbf{g}(A_i)$'s which is not a $q$-th power for any $q\in A_s$ so we take $m_s=1$. By induction suppose that  we have chosen $m_s,...m_{t+1}$ for some $1\leq t<s$, then  we have 
\begin{align}
\log \left(\prod_{i=t+1}^{s}g(A_i)^{m_i} \right) \not \equiv 0 \pmod{q}     \nonumber 
\end{align}
for all $q\in A_s\cup A_{s-1}...\cup A_{t+1}.$ Now, we observe that again by Remark \ref{rem2}  the element $\mathbf{g}(A_t)$ is $q$-th power residue for all $q\in A_t$. Thus we can apply Lemma \ref{mainlemma} with $A= A_t$, $B=A_s\cup... \cup A_{t+1}$, $y=g(A_t)$ and $z=\prod_{i=t+1}^{n}\mathbf{g}(A_i)^{m_i}$ to find some integer $m_t\in [1,j(P(A_t))]$ such that $\prod_{i=t}^{s}\mathbf{g}(A_i)^{m_i}$ is not a $q$-th power non-residue for all $q\in A_s\cup...\cup A_t$ and by induction the proof is concluded. 
\end{proof}	
Combining the Corollary with Propostion \ref{Vin} we obtain the following:  
\begin{prop}
	\label{main1}
	Let $p$ be a prime and let $A_1,A_2,...,A_s$ be defined as in (\ref{A's}), furthermore inductively define the divisors $d_1=\prod_{q|p-1}q$ and $d_i=d_{i-1}/\prod_{q\in A_{i-1}}q.$ Then for any $\epsilon>0$ we have the bound
	\begin{align}
		g(p)\ll \exp\left[\log p \left( \frac{1}{4\sqrt{e}} + \sum_{i=1}^{s-1} \frac{j(P_{A_i})}{u(d_i)}+ \epsilon\right)\right]. \nonumber
	\end{align}
	\begin{proof}
		Since $g(p)$ is the smallest element mod $p$ which is not a $q$-th power residue for all $q|p-1$, Corollary \ref{Cor 2.1} gives us the bound
		\begin{align}
		g(p)\leq \prod_{i=1}^{s-1}\mathbf{g}(A_i)^{m_i}\mathbf{g}(A_s)	\label{prop2,1}
		\end{align}
		where $m_i\leq j(P(A_i))$. Now we claim that $\mathbf{g}(A_i)=\mathbf{g}(d_i)$: indeed, $\mathbf{g}(d_i)$ is not a $q_i$-th power non residue for some $q_i|d_i$  and by minimality we have $\mathbf{g}(d_i)=\mathbf{g}(q_i)$; thus, by definition of the $A$'s, we must have $\mathbf{g}(q_i)=\mathbf{g}(A_j)$ for some $j$. Since $d_i$ does not contains prime divisors from the sets $A_j$ with $j<i$ we must have $j\geq i$, but $j>i$ cannot happen either because it would imply that $\mathbf{g}(d_i)>\mathbf{g}(A_i)$, contradicting minimality. So our claim is proved and together with (\ref{prop2,1})  we deduce that 
		\begin{align}
				g(p)\ll \prod_{i=1}^{s-1}\mathbf{g}(d_i)^{m_i}\mathbf{g}(d_s). \nonumber
		\end{align}  
		Now we apply Proposition \ref{Vin} (i) to $\mathbf{g}(d_i)$ for $1\leq i\leq s-1$ and bound the $m_i$ trivially to obtain 
	\begin{align}
		g(p)\leq \exp\left[\log p\left(\sum_{i=1}^{s-1}\frac{j(P_{A_i})}{u(d_i)}\right)\right]\mathbf{g}(d_s) .\label{prop2,2}
	\end{align}
	To bound $\mathbf{g}(d_s)$ we wish to apply Proposition \ref{Vin} (ii), so we split into two cases: if there exist some prime $q|d_s$ such that $u(q)> 4\sqrt{e}$ (where $u(\cdot)$ was defined in section 2) then $\mathbf{g}(d_s)\leq p^{1/4\sqrt{e}}$ by Proposition \ref{Vin} (i). Otherwise all the prime divisors of $d_s$ are uniformly bounded and we can apply Proposition \ref{Vin2} to deduce that for any $\epsilon>0$ we have $\mathbf{g}(d_s)\ll p^{1/4u(d_s) + \epsilon}$. Finally, since $u(\cdot)$ is decreasing and $d_s\geq 2$ we have $u(d_s)\geq u(2)=\sqrt{e}$ (see Remark \ref{rem1}). All in all we obtain
	
	\begin{align}
		\mathbf{g}(d_s)\leq \exp\left[\log p \left( \frac{1}{4\sqrt{e}}+ \epsilon\right)\right] \label{Prop2,3}.
	\end{align}
 Hence inserting (\ref{Prop2,3}) into (\ref{prop2,2}) we deduce the Proposition. 
	\end{proof}
\end{prop}
\section{Average behavior of Jacobsthal function}
\label{SJacobsthal}
In this section we study in more detail the Jacobsthal function in order to make explicit the bound in Proposition \ref{main1}. Jaconsthal himself observed that if $n$ is a positive integer such that for all primes $q$ dividing $n$ one has $q>\omega(n)$ then by the pigeon-principle follows that 
\begin{align}
	j(n)\leq \omega(n) \nonumber
\end{align}
We will show that a similar bound holds uniformly on the divisors $d$ of $p-1$  for \textquotedblleft most \textquotedblright primes $p$ when $p-1$ does not have small prime factors (but $2$), more precisely we prove that: 
\begin{prop}
	\label{main2}
	For any $\delta>0$ there exists some (large) $y=y(\delta)>0$ such that 
	\begin{align}
	\mathbb{P}_x\left(j(d)\leq 10\omega(d)\big\rvert \text{$p$ is $y$-rough}\right)>1-\delta \nonumber
	\end{align}
	uniformly for all divisors $d|p-1$. 
\end{prop}
In order to do this we will need to control the number of divisors $d|p-1$ in certain intervals for a generic prime $p$ and this is  the content of the next lemma which is certainly well-known but we could not find a precise reference. 

\begin{lem}
	\label{divisorofprimes}
	Let $t$ be a large parameter, then uniformly for $\xi>1$ we have  
	\begin{align}
		\mathbb{P}_x\left(|\omega(p-1,t)-\log\log t|\geq \xi \sqrt{\log\log t}\right)\ll \xi^{-2}  \nonumber
	\end{align}
	where the constant implied in the notation is absolute. 
\end{lem}
\begin{proof}
	First of all we observe that it is enough to prove the lemma for $t\leq x^{0.24}$, in fact if $t>x^{0.24}$ then 
	$|\omega(p-1,t)-\omega(p-1,x^{0.24})|\leq 5$ uniformly in $p$ and since $t$ is assumed to be large,  this difference is unimportant. So we begin by computing the first moment
	\begin{align}
			\mathbb{E}_x(\omega(p-1,t))=\frac{1}{\pi(x)}\sum_{p\leq x} \sum_{\substack{q|p-1\\ q \text{prime}}}1 
	\end{align}
	 Using the Bombieri Vinogradov Theorem \cite[Page 420]{IK} in the form $\sum_{q<t}\pi(x;q,1)=\sum_{q<t}\pi(x)/\varphi(q)+O(x/\log ^{2}x)$, the Prime Number Theorem $\pi(x)\sim x/\log x$ and switching the order of summation we obtain 
	\begin{align}
		\mathbb{E}_x(\omega(p-1,t))=\frac{1}{\pi(x)}	\sum_{q\leq t} \sum_{p\equiv 1 (q)}1= \sum_{q\leq t}\left(\frac{1}{\varphi(q)}+ \left(\frac{1}{\log x}\right)\right)= \log\log t +O(1) \label{mean}
	\end{align}
	Now  we can compute the second moment  expanding the square as  
	\begin{align}
	\mathbb{E}_x(\omega(p-1,t)^2)=\mathbb{E}_x\left(\sum_{\substack{q_1=q_2<t \\ q^2_1|p-1}}1 + \sum_{\substack{q_1\neq q_2 \\q_1<t , q_2<t \\ q_1q_2|p-1 }}1\right) \label{L3}
	\end{align}
Since $t\leq x^{0.24}$ we have $q_1\cdot q_2\leq x^{0.48}$ and so we can use Bombieri-Viogradov theorem again to see that 
\begin{align}
	\frac{1}{\pi(x)}\sum_{p\leq x}\sum_{\substack{q_1=q_2<t \\ q^2_1|p-1}}1=\sum_{q<t}\frac{1}{\varphi(q)^2}+o(1)=O(1) \label{L1}
\end{align}
and 
\begin{align}
\frac{1}{\pi(x)}	\sum_{\substack{q_1\neq q_2 \\q_1<t , q_2<t \\ q_1q_2|p-1 }}1= \sum_{q_1<t,q_2<t}\frac{1}{\varphi(q_1)\varphi(q_2)}=\log\log^2t+ O(\log\log t) \label{L2}
\end{align}
Inserting (\ref{L1}) and (\ref{L2}) into (\ref{L3}) and using (\ref{mean}) we can compute the variance of $\omega(p-1,t)$ as 
\begin{align}
	\mathbb{E}_x(\omega(p-1,t)^2)-\mathbb{E}_x(\omega(p-1,t))^2\ll \log\log t. \nonumber
\end{align}
Hence the Lemma follows at once by Chebyshev inequality. 
\end{proof}
\begin{cor}
	\label{primescor}
	Let $a$ be a (large) parameter then we have that
	\begin{align}
		\mathbb{P}_x\left(\underset{a<t<p-1}{\sup}\left|\frac{\omega(p-1,t)}{\log\log^2 t}\right|\ll 2\right)\ll\log\log^{-2}(a) \nonumber
	\end{align}
		where the constant implied in the notation is absolute. 
\end{cor}
\begin{proof}
	We introduce the check-points $t_j=\exp\exp j$ with $j\geq \log\log a$, then, bearing in mind that $\omega(p-1,t)$ in non-decreasing and it takes only integer values and that $|\log\log t_{j+1}-\log\log t_j|\leq 1$, we have that 
	\begin{align}
	&\underset{a<t<p-1}{\sup}\left|\frac{\omega(p-1,t)}{\log\log^2 t}\right|\geq 2 \hspace{15mm}\Longrightarrow  & \underset{a<t_j<p-1}{\sup}\left|\frac{\omega(p-1,t_j)}{j^2}\right|\geq 1 \nonumber
	\end{align}
	Now by Lemma \ref{divisorofprimes} with $\xi= j^{3/2}$ we see that we can bound the number of exemptions by 
	\begin{align}
	O\left(\sum_{\log\log a\leq j}j^{-3}\right)\ll \log\log^{-2}(a) \nonumber
	\end{align}
	as required. 
\end{proof}

 The next lemma partially follows the proof of the main theorem in \cite{E1}. 
\begin{lem}
	\label{jac}
	For all $\delta_1>0$ there exists some $y_1=y_1(\delta)$ such that 
	\begin{align}
		\mathbb{P}_x\left(j(d)\leq  \frac{2d}{\phi(d)}\omega(d)\Big\rvert \text{$p$ is $y_1$-rough}\right)>1-\delta_1 \nonumber
	\end{align}
	uniformly over the divisors $d$ of $p-1$.\footnote{The constants appearing in this and the following lemma are just chosen by the author and by no-means optimal, they could be substitute with any constant larger then 1 which is independent of all the parameters. }
\end{lem}
\begin{proof}
Let $\delta_1>0$ be given, suppose that $\omega(d)\leq \exp\exp(1/2\delta_1)$ and take $y_1=2(\exp\exp(1/2\delta_1)+ 1)$ then every prime factors (apart possibly $2$)  of $d$ is bigger then $2(\omega(d)+1)$ so it can divide at most one integer in every interval of length $2(\omega(d)+1)$, while $2$ divides at most half of them. Thus every interval of length $2(\omega(d)+1)$ contains at least one integer co-prime to $d$. 

We can therefore assume that $\omega(d)>\exp\exp(1/2\delta_1)$, for all primes $p$ and all divisors $d|p-1$. Now pick a divisor $d|p-1$ and let $z=2d\omega(d)/\phi(d)$, $I$ be any interval of length $z$ an
finally let $Q=\prod_{q|d, q<z}q$. By inclusion exclusion principle the number of integer co-prime to $Q$ in $I$ is 

\begin{align}
\sum_{\substack{n\in I \\ (n,Q)=1}}1=\sum_{n\in I}\sum_{d|(n,Q)}\mu(d)= \sum_{d|Q}\mu(d)\left\lfloor \frac{z}{d}\right\rfloor> z\frac{\phi(Q)}{Q}- 2^{\omega(Q)} \label{J1}
\end{align} 
where we have used the identity $\sum_{d|Q}\mu(d)/d=\phi(Q)/Q$. Now we apply Corollary \ref{primescor} with $a=\exp\exp(1/2\delta_1)$ to see that, uniformly on all divisors $d|p-1$, we have 
\begin{align}
\omega(Q)\leq 2 \log\log^2 z \label{J2}
\end{align}
with an exceptional set of primes of size at most $\delta_1$. On the other hand each prime divisor bigger then $z$  divide at most one integer in $I$ and thus the integers coprime to $d$ in $I$ are at least 
\begin{align}
2\omega(d) -2^{2\log\log^2z} -\omega(d)>0
\end{align}
since $z<2\omega(d)\log (\omega(d))$ and so $2\log\log^2z< \log \omega(d)$, this concludes the proof. 
\end{proof}

\begin{lem}
	\label{phi}
	For all $\delta_2>0$ there exists some $y_2=y_2(\delta_2)$ such that 
	\begin{align}
	\mathbb{P}_x\left(\frac{d}{\phi(d)}\leq 5\mid \text{$p$ is $y_2$-rough}\right)>1-\delta_2
	\end{align}
	uniformly for all divisors $d$ of $p-1$. 
\end{lem}
\begin{proof}
	Firstly we observe that for any divisor $d|p-1$
	\begin{align}
	\frac{d}{\phi(d)}=\prod_{\substack{q|d\\ q>y_2}}\left(1-\frac{1}{q}\right)^{-1}=\exp\left(\sum_{\substack{q|d\\ q>y_2}}\log\left(1-\frac{1}{q}\right)\right)=\exp\left(\sum_{\substack{q|d\\ q>y_2}}\frac{1}{q}+O\left(y_2^{-1}\right)\right) \nonumber
			\end{align}
			as the sum extended over the prime powers converges. So it is enough to show that for any $\delta_2>0$ there exists some $y_2$ such that 
			\begin{align}
				\mathbb{P}_x\left(\sum_{\substack{q|p-1\\ q>y_2}}\frac{1}{q}<\frac{\log 5}{2} \right)>1-\delta_2. \nonumber
			\end{align}
		To this end it is enough to estimate the first moment and the computation is similar to the one in Lemma \ref{divisorofprimes}. So we take expectation and we split the sum in $q\leq\sqrt{x}$ and $q>\sqrt{x}$. The latter (sum) contains at most $2$ elements and so can be bounded by $1/\sqrt{x}$, to estimate the former we switch the expectation and the sum and use Burn-Titchmarsh Theorem \cite[Page 167]{IK} (as we only need an upper bound) to obtain that 
		
		\begin{align}
			\mathbb{E}_x\left(\sum_{\substack{q|p-1\\ q>y_2}}\frac{1}{q}\right)&= 	\mathbb{E}_x\left(\sum_{\substack{q|p-1\\ \sqrt{x}>q>y_2}}\frac{1}{q}\right)+	\mathbb{E}_x\left(\sum_{\substack{q|p-1\\ q>\sqrt{x}}}\frac{1}{q}\right) \nonumber \\	
			 &=\frac{1}{\pi(x)}\sum_{y_2<q\leq \sqrt{x}}\frac{1}{q}\sum_{\substack{q|p-1\\}}1 +O\left(\frac{1}{\sqrt{x}}\right) \ll \sum_{y_2<q }\frac{1}{q\phi(q)}+\ll y_2^{-1}. \nonumber
		\end{align} 
		Hence by Markov inequality we obtain that 
			\begin{align}
				\mathbb{P}_x\left(\sum_{\substack{q|p-1\\ q>y_2}}\frac{1}{q}\geq\frac{\log 5}{2} \right)\leq \frac{2}{\log 5}\mathbb{E}\left(\sum_{\substack{q|p-1\\ q>y_2}}1/q\right)\ll y_{2}^{-1} \nonumber
			\end{align}
			so taking $y_2$ large enough in terms of $\delta_2$ we deduce the Lemma
		\end{proof}
	Combining the above two lemma we directly deduce the Proposition: 
\begin{proof}[Proof Proposition \ref{main2}]
	Pick $\delta_1=\delta/2$ in Lemma \ref{jac} and $\delta_2=\delta/2$ in Lemma \ref{phi} and $y=\max\{y_1,y_2\}$ (with $y_1$ and $y_2$ given by the respective lemma) then outside an event of probability $\delta_1+\delta_2=\delta$ and provided that $p$ is $y$-rough, we have both 
	\begin{align}
		&j(d)\leq 2 \frac{d}{\varphi(d)}\omega(d)& \frac{d}{\varphi(d)}\leq 5 \nonumber
	\end{align}
	uniformly on the divisors $d|p-1$ and the Proposition follows. 
\end{proof}

\section{Proof of Theorem \ref{Main theorem}}
\label{SMaintheorem}
The main theorem will follow directly from the next Proposition and a somehow technical lemma.   
\begin{prop}
	\label{main3}
	Let $p$ be a prime and let $q_1<...<q_{r}$ be the prime divisors of $p-1$, denote by $b_j=\prod_{i\leq j}q_i$ then we have that  for any $\epsilon_1>0$ and $\delta_1>0$ there exist some $y_1=y_1(\delta_1)$ such that 
	\begin{align}
		\mathbb{P}_x\left( g(p)\ll_{\epsilon} \exp\left[\log(p)\left(\frac{1}{4\sqrt{e}} +\sum_{j=2}^{r}\frac{10}{u(b_j)}+\epsilon_1\right)\right]\Big\rvert \text{$p$ is $y_1$-rough} \right)>1-\delta_1. \nonumber
	\end{align}
	Moreover the exceptional set consists of those primes failing condition $(1)$ in Theorem \ref{Main theorem}.
\end{prop}
\begin{proof}
	Let $\epsilon_1,\delta_1>0$ be given, then by Proposition \ref{main1} we know that 
	
	\begin{align}
		g(p)\ll \exp\left[\log p \left( \frac{1}{4\sqrt{e}} + \sum_{i=1}^{s-1} \frac{j(P_{A_i})}{u(d_i)}+ \epsilon_1\right)\right] \label{20}
	\end{align}
	where $A_i$ and $d_i$ are defined in the Proposition. Applying Proposition \ref{main2} with  $\delta=\delta_1$ we have, uniformly on the divisors $d|p-1$, the bound 
	\begin{align}
		j(d)\leq 10 \omega(d)  \nonumber
	\end{align}
	provided $p$ is $y_1$-rough. Therefore we can bound the sum in (\ref{20}) as  
	\begin{align}
		\sum_{i=1}^{s-1} \frac{j(P_{A_i})}{u(d_i)}\leq 	10\sum_{i=1}^{s-1} \frac{\omega(P_{A_i})}{u(d_i)} = 10\sum_{i=1}^{s-1} \frac{|A_i|}{u(d_i)} \label{21}
	\end{align}
	Now we observe that, since $u(\cdot)$ is decreasing, if some $A_i>1$ then 
	
	\begin{align}
	 \frac{|A_i|}{u(d_i)}< \frac{1}{u(d_i)}+\frac{|A_i|-1}{u(d_i/q)} \nonumber
	\end{align}
	for any $q|d_i$. Inserting this observation into  (\ref{21}) and recalling the definition of $d_i$, we then deduce  that 
	\begin{align}
		\sum_{i=1}^{s-1} \frac{j(P_{A_i})}{u(d_i)}\leq 10 \sum_{j=2}^{r}\frac{1}{u(b_j)} \label{22}
	\end{align}
	Hence (\ref{22}) holds outside the event $j(d)\leq 10 \omega(d)$ which has probability at most $\delta_1$ provided $p$ is $y_1$-rough and so combining with  (\ref{20}) we obtain the Proposition. 
\end{proof}

\begin{lem}
	\label{technical}
For any $\epsilon_2>0$ and $\delta_2>0$ there exists some  $y_2=y_2(\epsilon,\delta_2)$ such that 

\begin{align}
	\mathbb{P}_x\left(\sum_{\substack{q|p-1\\ q\neq 2}}\frac{\log\log q}{\log q}<\epsilon_2\Big\rvert\text{$p$ is $y_2$-rough}\right)>1-\delta_2. \nonumber
\end{align}
\end{lem}
\begin{proof}
Let $\epsilon_2,\delta_2>0$ be given, then the Lemma is equivalent to show that for all but $\delta_2$ primes up to $x$ we have the bound 
\begin{align}
\sum_{\substack{q|p-1\\ q>y_2}}\frac{\log\log q}{\log q}<\epsilon_2. \nonumber
\end{align}
The proof of this is very similar to the one of Lemma (\ref{phi}); so we firstly compute the expectation of the aforementioned sum, splitting into terms with $q<\sqrt{x}$ and $q>\sqrt{x}$. The letter (sum) contains at most 2 terms so can be bounded by $O(\log\log x/\log x)$,  to estimate former we switch the expectation and the sum and use Burn-Titchmarsh Theorem  to obtain that  
\begin{align}
	\mathbb{E}_x\left[\sum_{\substack{q|p-1\\q>y_2}}\frac{\log\log q}{\log q}\right]&= 	\mathbb{E}_x\left[\sum_{\substack{q|p-1\\\sqrt{x}>q>y_2}}\frac{\log\log q}{\log q}\right]+ 	\mathbb{E}_x\left[\sum_{\substack{q|p-1\\q>\sqrt{x}}}\frac{\log\log q}{\log q}\right]\nonumber \\
	&= \sum_{y_2<q<\sqrt{x}}\frac{\log\log q}{\log q}\sum_{q|p-1}1 +O\left(\frac{\log\log x}{\log x}\right) \ll \sum_{y_2<q<\sqrt{x}} \frac{\log\log q}{\varphi(q)\log q}.\nonumber
\end{align}
Now we observe that if we denote by $q_n$ the $n$-th prime the the Prime Number Theorem implies that $q_n\asymp n\log n$ and so we can bound the above sum as 
\begin{align}\sum_{y_2<q<\sqrt{x}} \frac{\log\log q}{\varphi(q)\log q}\ll \sum_{n_y<n}\frac{\log\log n}{n\log^2n}\ll \frac{\log\log n_y}{\log n_y}. \nonumber
\end{align}
where $n(y_2)$ is defined so that $q_{n(y_2)}$ is the smallest prime bigger then $y_2$. By Markov's inequality we then obtain that 
\begin{align}
		\mathbb{P}_x\left(\sum_{\substack{q|p-1\\q>y_2}}\frac{\log\log q}{\log q}\geq \epsilon_2\right)\leq \epsilon_2^{-1}\mathbb{E}\left[\sum_{q>y_2}\frac{ \log\log q}{\log q}\right]\ll \epsilon_2^{-1}\frac{\log\log n(y_2)}{\log n(y_2)} \nonumber
\end{align}
Hence taking $n(y_2)$ large enough in terms of $\epsilon_2$ and $\delta_2$ we obtain the Lemma. 
\end{proof}
We are finally ready to prove the main result:  
\begin{proof}[proof of Theorem \ref{Main theorem}] Let $\epsilon>0$ and $\delta>0$ be given; take $\delta_1=\delta/2$ and $\epsilon_1=\epsilon/2$ in Proposition \ref{main3}, provided that $p$ is $y_1$-rough (with $y_1$ given in the Proposition) we  have  
	 \begin{align}
	 g(p)\ll \exp\left[\log(p)\left(\frac{1}{4\sqrt{e}}
	  +\sum_{j=2}^{r}\frac{10}{u(b_j)}+\epsilon/2\right)\right] \nonumber
	  \end{align}
	a part from the primes failing condition $(1)$. Then, by Corollary \ref{corVin}, bearing in mind the definition of $b_i$ given in the Proposition,  we can bound the above\footnote{Here we lower bound $u(b_j)$ by $u(q)$ where $q$ is the large prime $q|b_j$. Considering $u(b_j)$ instead would make the computation much more messy and change the value of the sum only by a constant, as it can be seen by the fact that for almost all primes if $q_j$ is $j$-th prime divisor of $p-1$ then $\log\log q_j \approx j$. }  as 
	  
	  \begin{align}
	g(p)\ll \exp\left[\log(p)\left(\frac{1}{4\sqrt{e}}
	+\sum_{\substack{q|p-1 \\ q\neq 2}}\frac{10 \cdot \log\log q}{\log q}+\epsilon/2\right)\right]. \label{23}
		\end{align}
	Now pick $\delta_2=\delta/2$ and $\epsilon_2=\epsilon/20$ in Lemma \ref{technical} and let $y=\max\{y_1,y_2\}$ (with $y_2$ given by the lemma) then outside the event defined by condition $(1)$ and $(2)$ which has size at most $\delta_1+\delta_2=\delta$ and provided that $p$ is $y$-rough we have both (\ref{23}) and 
	\begin{align}
		\sum_{\substack{q|p-1 \\ q\neq 2}} \frac{10\log\log q}{\log q}\leq \epsilon_2/2 \label{24}
	\end{align}
	and the Theorem follows. 
\end{proof}
The proof of the Corollary then follows by a simple application of inclusion-exclusion principle:
\begin{proof}[proof of Corollary \ref{Main Corollary}]
We firstly observe that given two events $A,B$ then 
\begin{align}
	\mathbb{P}(A\cap B)=\mathbb{P}(A|B)\mathbb{P}(B). \nonumber
\end{align} 
Let $\epsilon>0$ be given, take $\delta=\epsilon$  and let $y=y(\epsilon)$ be given by the Theorem. Now  take $A$ to be the event $p$ satisfies conditions $(1)$, $(2)$ and $B$ the event $p$ is $y$-rough, then the Theorem tells us that $P(A|B)>1-\epsilon$ so it is enough to show that $P(B)$ is no-zero.  By the Prime Number Theorem in arithmetic progressions we know that for any fixed $d$ the number of primes up to $x$  congruent to $1$ mod $d$ is $\pi(x)/\varphi(d) + O(x/\log^2x)$ and so, by the inclusion exclusion principle, we obtain that   
\begin{align}
\mathbb{P}_x(B)=\frac{1}{\pi(x)}\sum_{d|P(y)}\mu(d)\sum_{\substack{p\leq x \\d|p-1}}1=  \sum_{d|P(y)}\frac{\mu(d)}{\varphi(d)} + O\left(\frac{1}{\log x}\right) \nonumber
\end{align}
where $P(y)=\prod_{2<q\leq y}q$.  The above sum can be estimated using Merten's Theorem as 
\begin{align}
	 \sum_{d|P(y)}\frac{\mu(d)}{\varphi(d)}= \prod_{2<q\leq y}\left(1-\frac{1}{\varphi(q)}\right)\asymp\log^{-1} y \nonumber
\end{align}
which is non-zero and hence the first part of the Corollary follows. The second part is a direct consequence of the first and the Theorem. 
\end{proof}

\section*{Acknowledgement}
The author would like to thank Andrew Granville  for his help and guidance in writing this manuscript and Igor Shparlinski for his comments on an early version of the work. This work was supported by the Engineering and Physical Sciences Research Council [EP/L015234/1].  
The EPSRC Centre for Doctoral Training in Geometry and Number Theory (The London School of Geometry and Number Theory), University College London.

	\end{document}